\documentclass[a4paper,11pt]{article}

\usepackage{amsmath}
\usepackage{amssymb}
\usepackage{amsthm}
\usepackage{a4wide}
\usepackage{xcolor}
\usepackage{graphicx}
\usepackage{subfigure}
\usepackage[lined,ruled]{algorithm2e}
\usepackage{tabularx}
\usepackage{makecell}
\usepackage{natbib}
\usepackage{appendix}

\SetKwProg{Fn}{function}{}{end}
\SetKwInOut{In}{input}
\SetKwInOut{Out}{output}
\SetKw{Or}{or}
\SetKw{And}{and}

\newtheorem{theorem}{Theorem}[section]

\theoremstyle{remark}
\newtheorem{remark}{Remark}[section]
\newtheorem{example}{Example}[section]

\newcommand{\trig}[1]{{#1}^T}
\newcommand{\hypb}[1]{{#1}^H}
\newcommand{\gene}[1]{{#1}^X}
\newcommand{\ii}{{\rm i}}

\newcommand{\qbinom}{\genfrac{[}{]}{0pt}{}}

\makeatletter
\newcommand{\thealgorithm}{\arabic\algocf@float}

\newcommand{\AlgoCaptionFormat}{}

\renewcommand{\algocf@makecaption@ruled}[2]{%
  \global\sbox\algocf@capbox{\hskip\AlCapHSkip%
    \setlength{\hsize}{\columnwidth}
    \addtolength{\hsize}{-2\AlCapHSkip}
    \vtop{\AlgoCaptionFormat\algocf@captiontext{#1}{#2}}}
}%
\makeatother

\setcellgapes{3pt}

\begin{document}

\title{On the normalization of trigonometric and hyperbolic B-splines}

\author{Hendrik Speleers}

\date{\small Department of Mathematics, University of Rome Tor Vergata, Rome, Italy \\
E-mail: \texttt{speleers@mat.uniroma2.it}
}

\maketitle

\begin{abstract}
Trigonometric and hyperbolic B-splines can be computed via recurrence relations analogous to the classical polynomial B-splines. However, in their original formulation, these two types of B-splines do not form a partition of unity and consequently do not admit the notion of control polygons with the convex hull property for design purposes. In this paper, we look into explicit expressions for their normalization and provide a recursive algorithm to compute the corresponding normalization weights.
As example application, we consider the exact representation of a circle in terms of $C^{2n-1}$ trigonometric B-splines of order $m=2n+1\geq3$, with a variable number of control points. We also illustrate the approximation power of trigonometric and hyperbolic splines.
\end{abstract}


\section{Introduction}

Smooth polynomial splines are a practical and widely used tool in many applications. This is mainly thanks to their representation in terms of a B-spline basis, which can be computed efficiently and accurately by means of a stable recurrence relation \cite{Boor:2001}. These basis functions are nonnegative, have local support, and can be normalized to form a partition of unity.
The degrees of freedom in such representations can be geometrically interpreted as control points, leading to the notion of control polygons. The (local) convex hull property makes them very useful and intuitive for computer-aided geometric design of parametric curves \cite{CohenRE:2001,PieglT:2012}.
For the description of conic sections, one often relies on their rational form, called NURBS, which are essentially B-splines in homogeneous coordinates. However, rational curves may be unstable and their derivatives are difficult to compute \cite{MainarPS:2001}.

B-splines (and NURBS) are also the driving force behind isogeometric analysis, a methodology for the analysis of problems governed by partial
differential equations \cite{CottrellHB:2009}. Isogeometric analysis aims to simplify the interoperability between geometric modeling and numerical simulation by constructing a fully integrated framework for computer-aided design and finite element analysis, based on spline representations. The isogeometric paradigm shows important advantages over classical finite element analysis. Firstly, complicated geometries can be represented more accurately and some common profiles in engineering, such as conic sections, are described exactly. Secondly, the inherently high smoothness of B-splines leads to a better accuracy per degree of freedom and improved spectral properties.
It has been shown that maximally smooth splines on uniform partitions are close to optimal (in the sense of Kolmogorov $n$-widths) from an approximation point of view \cite{SandeMS:2019}.

Besides the well-known (polynomial) splines mentioned above, there is a large zoo of generalized splines, including Tchebycheffian splines, L-splines, and many others; see \cite{Schumaker:2007} and references therein. B-spline-like bases have been constructed for a variety of these generalized spline spaces. Yet, two classes of trigonometric and hyperbolic splines are particularly closely linked to polynomial splines \cite{Schumaker:1982} because there exist analogous recurrence relations for the  computation of their bases \cite{LycheW:1979,Schumaker:1983}. The corresponding B-splines, unfortunately, do not form a partition of unity and consequently no convex hull property is available. As a remedy, other design mechanisms have been developed for them. In \cite{KochLNS:1995}, a trigonometric version of the convex hull property has been obtained by introducing control curves.
These classes of splines are well suited for building high-quality parametric descriptions of conic sections and cam profiles \cite{NeamtuPS:1998}. The availability of trigonometric and exponential functions in the space is also beneficial in the context of isogeometric analysis \cite{RavalMS:2023}.

Since the use of control polygons and their convex hull property are so deeply rooted in classical computer-aided geometric design, it is natural to ask whether they can also be established for trigonometric and hyperbolic splines. This brings us to the problem of rescaling the corresponding B-spline bases such that they form a partition of unity, preferably with explicit and simple expressions. By definition of the spaces, this is only possible for such splines of odd order. Indeed, spaces of even order do not contain constants. In the original papers \cite{LycheW:1979,Schumaker:1983}, Marsden-like identities have been provided for the representation of trigonometric and hyperbolic power functions in terms of the corresponding B-splines, with the representation of constants as a special case. However, those identities were not fully explicit. Later, in \cite{LycheSS:1998}, an explicit formulation has been derived for trigonometric splines using trigonometric polar forms. A similar strategy can also be applied to hyperbolic splines.

The adoption of normalized trigonometric B-splines (of odd order) has been recognized in \cite{NeamtuPS:1998} and the associated normalization weights were explicitized for a general (nonuniform) knot partition.
An alternative explicit expression for normalized trigonometric B-splines on a uniform knot partition can be found in \cite{Walz:1997}.
The particular trigonometric cases of order $3$ and $5$ with uniform knots have recently been studied 
in \cite{AlbrechtMPR:2023a,AlbrechtMPR:2023b}.
The simplified setting, where the partition only consists of a single interval, has also received considerable attention; see, e.g., \cite{Han:2015,SanchezReyes:1998} for the trigonometric case and \cite{ShenW:2005} for the hyperbolic case.
The normalization makes that the classical convex hull property is regained and can be exploited for design and analysis.

Unfortunately, even though they are elegant, the explicit formulas for the normalization weights of the considered B-splines are computationally expensive: each of them requires the calculation of a sum, see \eqref{eq:B-spline-trig-weights-full} and \eqref{eq:B-spline-hypb-weights-full}, whose number of terms $|\mathcal{Q}_n|$ has a strong factorial growth in the spline order $m=2n+1$ (the growth of $|\mathcal{Q}_n|$ is depicted in Figure~\ref{fig:card-set}). Therefore, they are less appealing in a practical implementation and especially in the context of isogeometric analysis, where the use of high-order splines is common.

Recently, in the more general context of multi-degree Tchebycheffian B-splines (MDTB-splines), a Matlab toolbox for the computation of normalized B-splines has been developed in \cite{Speleers:2022}, with specialized routines for the trigonometric and hyperbolic cases. These routines do not utilize the existing recurrence relations for B-splines but are based on a so-called spline extraction operator \cite{HiemstraHMST:2020}. A similar extraction operator has also been exploited in \cite{SpeleersT:2021} for obtaining smooth piecewise NURBS descriptions of circles and ellipses.

In this paper, we continue the investigation into the construction of normalized trigonometric and hyperbolic B-splines, with a particular focus on the practical computation of normalization weights for high orders.
More precisely, we derive new explicit expressions for such weights and provide a practical method to compute them. In Section~\ref{sec:definition}, we start by recalling the recurrence relations for trigonometric/hyperbolic B-splines and the known explicit expressions for their normalization weights. Then, in Section~\ref{sec:weights}, we state our main results: the new expressions of the weights (see Theorem~\ref{thm:weights}) and how to compute them recursively (see Algorithm~\ref{alg:weights}). We also provide a simplified version in the case of uniform knots (see Theorem~\ref{thm:weights-uniform}).
In Section~\ref{sec:examples}, we apply the new expressions in the construction of normalized trigonometric/hyperbolic B-splines and consider the exact representation of a circle in terms of $C^{2n-1}$ trigonometric B-splines of order $m=2n+1\geq3$, with a variable number of control points.
We also illustrate the approximation power of trigonometric and hyperbolic splines.
In Section~\ref{sec:conclusion}, we conclude with a summary of our results.

\section{Definition of normalized B-splines}\label{sec:definition}

In this section, we recall from the literature the recurrence relations for trigonometric B-splines \cite{LycheW:1979} and hyperbolic B-splines \cite{Schumaker:1983}. In their standard form, these B-splines do not form a partition of unity, but it is possible to normalize them (for odd order).

Given a positive integer $m$, suppose that $\{x_j\}$ is a nondecreasing sequence of real numbers, called knots, satisfying $x_{j+m}-x_j<\pi$ for all $j$.
Trigonometric B-splines of order $m$ are usually defined by means of a stable recurrence relation. Following the construction in \cite{LycheW:1979}, they can be computed as
$$
T_{j,1}(x)=\begin{cases}
1\Big/\sin\left(\dfrac{x_{j+1}-x_j}{2}\right) & \text{if } x_j\leq x< x_{j+1},\\
0 & \text{otherwise},
\end{cases}
$$
and for $m>1$,
\begin{equation}\label{eq:recursion-trig}
T_{j,m}(x)=\dfrac{\sin\left(\dfrac{x-x_j}{2}\right)}{\sin\left(\dfrac{x_{j+m}-x_j}{2}\right)}T_{j,m-1}(x)
+\dfrac{\sin\left(\dfrac{x_{j+m}-x}{2}\right)}{\sin\left(\dfrac{x_{j+m}-x_j}{2}\right)}T_{j+1,m-1}(x)
\end{equation}
if $x_{j+m}>x_j$, and $T_{j,m}(x)=0$ if $x_{j+m}=x_j$.

Let us fix $m=2n+1$ ($n\geq0$ integer) and let $\mathcal{Q}_n$ be the set of vectors $\boldsymbol{q}=(q_1,\ldots,q_{2n})\in\mathbb{Z}^{2n}$ that consists of all permutations of
$$
(1,2,\ldots,2n).
$$
It is clear that the cardinality of this set equals
$$
|\mathcal{Q}_n|=(2n)!.
$$
Then, we rescale the trigonometric B-splines of odd order $m$ as
\begin{equation}\label{eq:B-spline-trig-def}
\trig{N}_{j,m}(x)=\trig{w}_{j,m}\sin\left(\dfrac{x_{j+m}-x_j}{2}\right) T_{j,m}(x),
\end{equation}
where
$$
\trig{w}_{j,1}=1,
$$
and for $m=2n+1\geq3$,
\begin{equation}\label{eq:B-spline-trig-weights-full}
\trig{w}_{j,m} = \dfrac{1}{|\mathcal{Q}_n|}
\sum_{\boldsymbol{q}\in \mathcal{Q}_n}\prod_{k=1}^{n}\cos\left(\frac{x_{j+q_{2k}}-x_{j+q_{2k-1}}}{2}\right).
\end{equation}

It has been shown in \cite{LycheW:1979} that the trigonometric B-splines $T_{j,m}$ are nonnegative and have local support, i.e.,
$$
T_{j,m}(x)>0, \quad x\in(x_j,x_{j+m}),
$$
and
$$
T_{j,m}(x)=0, \quad x\not\in[x_j,x_{j+m}).
$$
Since $x_{j+m}-x_j<\pi$, it is easy to check that these properties carry over to the normalized versions $\trig{N}_{j,m}$.
As observed in \cite{NeamtuPS:1998}, the functions in \eqref{eq:B-spline-trig-def} form a (local) partition of unity. More precisely, assuming $m\leq l\leq r$ and $x_l<x_{r+1}$, we have
$$
\sum_{j=l+1-m}^{r}\trig{N}_{j,m}(x)=1, \quad x_l\leq x<x_{r+1}.
$$
This follows directly from the results in \cite[Section~5]{LycheSS:1998} based on trigonometric polar forms. Furthermore, it has been shown in \cite{LycheW:1979} that the trigonometric B-splines of odd order $m=2n+1$ are splines with pieces belonging to the space
\begin{equation}\label{eq:local-space-trig}
\left\langle 1, \cos(x), \sin(x), \cos(2x), \sin(2x), \ldots, \cos(nx), \sin(nx) \right\rangle.
\end{equation}
At a given knot, they have smoothness $C^{2n-\mu}$, where $\mu$ is the multiplicity of the knot in the considered knot sequence.
Trigonometric splines were introduced in \cite{Schoenberg:1964}.

\begin{remark}\label{rmk:even-trig}
In this paper, with the purpose of normalization, we are only interested in the trigonometric B-splines $T_{j,2n+1}$ of odd order (see \eqref{eq:B-spline-trig-def}), which belong piecewisely to the space \eqref{eq:local-space-trig}.
When applying the recursion \eqref{eq:recursion-trig}, we also obtain the (intermediate) B-splines $T_{j,2n}$ of even order, which belong piecewisely to the space
$$
\left\langle \cos(\tfrac{1}{2}x), \sin(\tfrac{1}{2}x), \cos(\tfrac{3}{2}x), \sin(\tfrac{3}{2}x), \ldots, \cos((n-\tfrac{1}{2})x), \sin((n-\tfrac{1}{2})x) \right\rangle.
$$
However, such a space does not contain constants and no rescaling to form a partition of unity is possible for even order.
\end{remark}

\begin{remark}\label{rmk:knots-trig}
In the context of non-piecewise functions invariant under translations, the existence of bases with shape-preserving properties depends on the length of the interval domain, which must not exceed a certain ``critical length for design'' to ensure design-relevant properties; see, e.g., \cite{BeccariGM:2020,CarnicerMP:2003,Mazure:2005,Speleers:2022} and references therein. A similar restriction appears in the case of trigonometric B-splines, where the distance between knots must remain below a specific threshold.
The knot distance requirement $x_{j+m}-x_j<\pi$ considered here is stricter than in the paper \cite{LycheW:1979}. This ensures nonnegativity of the normalized function $\trig{N}_{j,m}$. It should be mentioned, however, that the requirement is only a sufficient condition but not a necessary condition for nonnegativity, as illustrated in Example~\ref{ex:circle}.
Actually, the weight $\trig{w}_{j,m}$ in \eqref{eq:B-spline-trig-weights-full} is positive if $x_{j+2n}-x_{j+1}<\pi$.
Without the normalization, the requirement for $T_{j,m}$ can be relaxed to $x_{j+m}-x_j<2\pi$.
\end{remark}

Now, we consider the hyperbolic case.
Suppose that $\{x_j\}$ is a nondecreasing sequence of real numbers.
From \cite{Schumaker:1983} we know that hyperbolic B-splines of order $m$ can be computed as
$$
H_{j,1}(x)=\begin{cases}
1\Big/\sinh\left(\dfrac{x_{j+1}-x_j}{2}\right) & \text{if } x_j\leq x< x_{j+1},\\
0 & \text{otherwise},
\end{cases}
$$
and for $m>1$,
\begin{equation}\label{eq:recursion-hypb}
H_{j,m}(x)=\dfrac{\sinh\left(\dfrac{x-x_j}{2}\right)}{\sinh\left(\dfrac{x_{j+m}-x_j}{2}\right)}H_{j,m-1}(x) +\dfrac{\sinh\left(\dfrac{x_{j+m}-x}{2}\right)}{\sinh\left(\dfrac{x_{j+m}-x_j}{2}\right)}H_{j+1,m-1}(x)
\end{equation}
if $x_{j+m}>x_j$, and $H_{j,m}(x)=0$ if $x_{j+m}=x_j$.

Fix again $m=2n+1$ ($n\geq0$ integer). A partition of unity is achieved by rescaling the hyperbolic B-splines as
\begin{equation}\label{eq:B-spline-hypb-def}
\hypb{N}_{j,m}(x)=\hypb{w}_{j,m}\sinh\left(\dfrac{x_{j+m}-x_j}{2}\right) H_{j,m}(x),
\end{equation}
where
$$
\hypb{w}_{j,1}=1,
$$
and for $m=2n+1\geq3$,
\begin{equation}\label{eq:B-spline-hypb-weights-full}
\hypb{w}_{j,m} = \dfrac{1}{|\mathcal{Q}_n|}
\sum_{\boldsymbol{q}\in \mathcal{Q}_n}\prod_{k=1}^{n}\cosh\left(\frac{x_{j+q_{2k}}-x_{j+q_{2k-1}}}{2}\right).
\end{equation}

Similar to the trigonometric case, the hyperbolic B-splines $H_{j,m}$ and $\hypb{N}_{j,m}$ are nonnegative and have local support; see \cite{Schumaker:1983}. Furthermore, they are splines with pieces belonging to
$$
\left\langle 1, \cosh(x), \sinh(x), \cosh(2x), \sinh(2x), \ldots, \cosh(nx), \sinh(nx) \right\rangle.
$$

\begin{remark}
When applying the recursion \eqref{eq:recursion-hypb}, the (intermediate) B-splines $H_{j,2n}$ of even order belong piecewisely to the space
$$
\left\langle \cosh(\tfrac{1}{2}x), \sinh(\tfrac{1}{2}x), \cosh(\tfrac{3}{2}x), \sinh(\tfrac{3}{2}x), \ldots, \cosh((n-\tfrac{1}{2})x), \sinh((n-\tfrac{1}{2})x) \right\rangle.
$$
As in the trigonometric case (see Remark~\ref{rmk:even-trig}), even order does not admit a rescaling to form a partition of unity and is of no interest in this paper.
\end{remark}

\section{Computation of the normalization weights}\label{sec:weights}

The explicit expressions of the normalization weights in \eqref{eq:B-spline-trig-weights-full} and \eqref{eq:B-spline-hypb-weights-full} require the summation of many products of cosine and hyperbolic cosine values, respectively. In this section, we propose a less expensive and more practical way to compute these weights.

Let $\mathcal{S}_n$ ($n\geq1$ integer) be the set of vectors $\boldsymbol{s}=(s_1,\ldots,s_{2n-1})\in\{-1,1\}^{2n-1}$ that consists of all permutations of
\begin{equation}\label{eq:definition-s}
(\underbrace{-1,\ldots,-1}_{n-1 \text{ times}},\underbrace{1,\ldots,1}_{n \text{ times}}).
\end{equation}
Note that the number of vectors in this set equals half the $n$th central binomial coefficient,
$$
|\mathcal{S}_n|=\binom{2n-1}{n-1}=\frac{1}{2}\binom{2n}{n}.
$$

\begin{theorem}\label{thm:weights}
Let $m=2n+1\geq3$.
The normalization weight in \eqref{eq:B-spline-trig-weights-full} can be expressed as
\begin{equation}\label{eq:B-spline-trig-weights}
\trig{w}_{j,m}=\dfrac{1}{|\mathcal{S}_n|}
\sum_{\boldsymbol{s}\in \mathcal{S}_n}\cos\left(\frac{-x_{j+1}+\sum_{k=1}^{2n-1}s_{k} x_{j+k+1}}{2}\right),
\end{equation}
and similarly, the normalization weight in \eqref{eq:B-spline-hypb-weights-full} can be expressed as
\begin{equation}\label{eq:B-spline-hypb-weights}
\hypb{w}_{j,m}=\dfrac{1}{|\mathcal{S}_n|}
\sum_{\boldsymbol{s}\in \mathcal{S}_n}\cosh\left(\frac{-x_{j+1}+\sum_{k=1}^{2n-1}s_{k} x_{j+k+1}}{2}\right).
\end{equation}
\end{theorem}
\begin{proof}
We only consider the trigonometric case (the proof of the hyperbolic case is analogous).
Let us first observe that \eqref{eq:B-spline-trig-weights-full} can be regarded as taking the average of the cosine products over the set $\mathcal{Q}_n$.
Plugging the identity
\begin{equation}\label{eq:cos-exp}
\cos(x)=\frac{\exp(\ii x)+\exp(-\ii x)}{2},
\end{equation}
with $\ii=\sqrt{-1}$, into the expression \eqref{eq:B-spline-trig-weights-full} yields
$$
\trig{w}_{j,m} = \dfrac{1}{(2n)!}
\sum_{\boldsymbol{q}\in \mathcal{Q}_n}\prod_{k=1}^{n}\exp\left(\ii\frac{x_{j+q_{2k}}-x_{j+q_{2k-1}}}{2}\right).
$$
This equals to
$$
\trig{w}_{j,m} = \dfrac{1}{(2n)!}
\sum_{\boldsymbol{q}\in \mathcal{Q}_n}\exp\left(\frac{\ii}{2}\sum_{k=1}^{n}x_{j+q_{2k}}-\frac{\ii}{2}\sum_{k=1}^{n} x_{j+q_{2k-1}}\right).
$$
Due to the summation over all permutations in the set $\mathcal{Q}_n$, several of these terms are identical. Indeed, for a given vector $\boldsymbol{q}$ it is clear that any permutation of the even-indexed elements will not affect the final contribution, and the same for the odd-indexed elements. There are $(n!)^2$ of such vectors and we can group them together.
After pruning the sum in this way, we arrive at
\begin{equation}\label{eq:B-spline-trig-weights-bis}
\trig{w}_{j,m} = \frac{1}{\binom{2n}{n}} \sum_{\substack{1\leq q_1<q_2<\cdots \\ \hfill <q_n\leq 2n}}
\exp\left(\ii\sum_{k=1}^n x_{j+q_k} - \frac{\ii}{2}\sum_{k=1}^{2n} x_{j+k}\right).
\end{equation}
Finally, by taking into account the definition of the set $\mathcal{S}_n$ and the identity \eqref{eq:cos-exp}, the expression in \eqref{eq:B-spline-trig-weights-bis} can be rewritten into \eqref{eq:B-spline-trig-weights}.
\end{proof}

\begin{example}
Elaborating the normalization weight in \eqref{eq:B-spline-trig-weights} for $m=3,5$ ($n=1,2$) gives
$$
\trig{w}_{j,3}=\cos\left(\dfrac{-x_{j+1}+x_{j+2}}{2}\right),
$$
and
$$
\begin{aligned}
\trig{w}_{j,5}=\frac{1}{3}&\left[\cos\left(\frac{-x_{j+1}-x_{j+2}+x_{j+3}+x_{j+4}}{2}\right)\right.\\
&+\cos\left(\frac{-x_{j+1}+x_{j+2}-x_{j+3}+x_{j+4}}{2}\right)\\
&+\left.\cos\left(\frac{-x_{j+1}+x_{j+2}+x_{j+3}-x_{j+4}}{2}\right)\right].
\end{aligned}
$$
\end{example}

\begin{example}
Elaborating the normalization weight in \eqref{eq:B-spline-hypb-weights} for $m=3,5$ ($n=1,2$) gives
$$
\hypb{w}_{j,3}=\cosh\left(\dfrac{-x_{j+1}+x_{j+2}}{2}\right),
$$
and
$$
\begin{aligned}
\hypb{w}_{j,5}=\frac{1}{3}&\left[\cosh\left(\frac{-x_{j+1}-x_{j+2}+x_{j+3}+x_{j+4}}{2}\right)\right.\\
&+\cosh\left(\frac{-x_{j+1}+x_{j+2}-x_{j+3}+x_{j+4}}{2}\right)\\
&+\left.\cosh\left(\frac{-x_{j+1}+x_{j+2}+x_{j+3}-x_{j+4}}{2}\right)\right].
\end{aligned}
$$
\end{example}

The expressions of the normalization weights $\trig{w}_{j,m}$ and $\hypb{w}_{j,m}$ in \eqref{eq:B-spline-trig-weights} and \eqref{eq:B-spline-hypb-weights} can be calculated using the method described in Algorithm~\ref{alg:weights}. Each vector of the set $\mathcal{S}_n$ can be built recursively by adding either $-1$ or $1$ as next element in the vector, as long as their total so far is less than $n-1$ or $n$, respectively. Their number is stored as $\ell_-$ and $\ell_+$, respectively. The partial sum of knots is stored as $y=-x_{j+1}+\sum_{i=1}^{k-1}s_{i} x_{j+i+1}$, with $k=\ell_-+\ell_++1$.

\begin{algorithm}[t!]
\caption{Computation of the normalization weight $\gene{w}_{j,m}$, $X\in\{T,H\}$, according to \eqref{eq:B-spline-trig-weights} and \eqref{eq:B-spline-hypb-weights} for any $j$. Let $m=2n+1$ ($n\geq0$ integer) and suppose that $\{x_j\}$ is a nondecreasing sequence of real numbers. We set $c^T(x)=\cos(x)$ and $c^H(x)=\cosh(x)$.}\label{alg:weights}
\medskip
\Fn{$\operatorname{weight}(j)$}{
\eIf{$n = 0$}{
    \Return 1\;
}{
    $f_n \gets \prod_{i=1}^{n-1} i/(n+i)$\;
    $\gene{w}_{j,2n+1} \gets f_n \cdot \operatorname{weight\_rec}(j, 1, 0, 0, -x_{j+1})$\;
    \Return $\gene{w}_{j,2n+1}$\;
}
}
\medskip
\Fn{$\operatorname{weight\_rec}(j, k, \ell_-, \ell_+, y)$}{
\eIf{$k = 2n$}
{
    \Return $\gene{c}(y/2)$\;
}{
    $w \gets 0$\;
    \If{$\ell_- < n-1$}
    {
        $w \gets w + \operatorname{weight\_rec}(j, k + 1, \ell_- + 1, \ell_+, y - x_{j + k + 1})$\;
    }
    \If{$\ell_+ < n$}
    {
        $w \gets w + \operatorname{weight\_rec}(j, k + 1, \ell_-, \ell_+ + 1, y + x_{j + k + 1})$\;
    }
    \Return $w$\;
}
}
\end{algorithm}

It can be verified that
$$
|\mathcal{S}_n| = 
\frac{|\mathcal{Q}_n|}{2(n!)^2},
$$
which indicates that the number of terms in \eqref{eq:B-spline-trig-weights} and \eqref{eq:B-spline-hypb-weights} grows substantially slower in $n$ compared to \eqref{eq:B-spline-trig-weights-full} and \eqref{eq:B-spline-hypb-weights-full}, respectively. The cardinality of both sets is depicted in Figure~\ref{fig:card-set} for $n=1,\ldots,10$. Therefore, the proposed approach is far more efficient than computing the sums in \eqref{eq:B-spline-trig-weights-full} and \eqref{eq:B-spline-hypb-weights-full}.

\begin{figure}[t!]
\centering
\includegraphics[height=6cm]{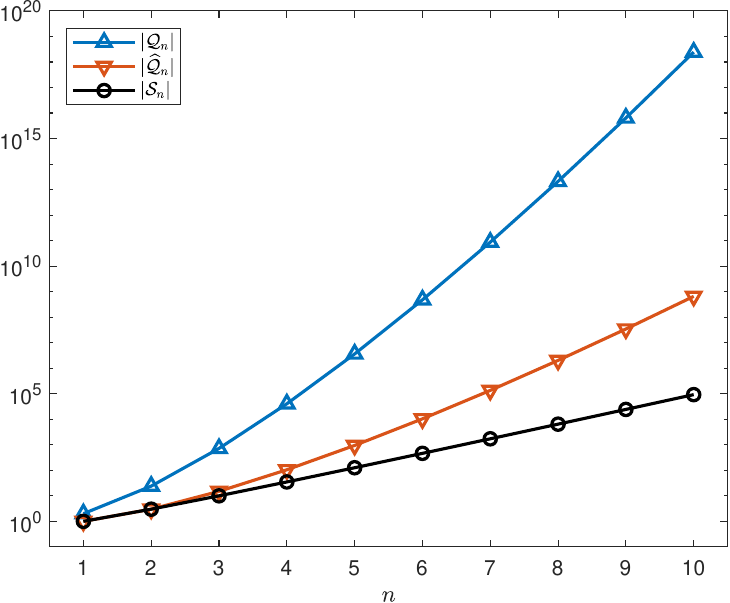}
\caption{Cardinality of the sets $\mathcal{Q}_n$, $\widehat{\mathcal{Q}}_n$, $\mathcal{S}_n$ for different values of $n$ (in logarithmic scale).}\label{fig:card-set}
\end{figure}

\begin{remark}
Since several terms in the sum of \eqref{eq:B-spline-trig-weights-full} are identical, there is quite a lot of redundancy in that expression and, after direct pruning, one obtains
\begin{equation}\label{eq:B-spline-trig-weights-tris}
\trig{w}_{j,m} = \dfrac{1}{|\widehat{\mathcal{Q}}_n|}
\sum_{\boldsymbol{q}\in \widehat{\mathcal{Q}}_n}\prod_{k=1}^{n}\cos\left(\frac{x_{j+q_{2k}}-x_{j+q_{2k-1}}}{2}\right)
\end{equation}
for $m=2n+1\geq3$, where $\widehat{\mathcal{Q}}_n\subset\mathcal{Q}_n$ is the set of vectors $\boldsymbol{q}=(q_1,\ldots,q_{2n})\in\mathbb{Z}^{2n}$ that consists of all permutations of
$$
(1,2,\ldots,2n)
$$
such that
$$
q_1<q_3<\cdots<q_{2n-1},
$$
and
$$
q_{2k-1}<q_{2k}, \quad k=1,2,\ldots,n.
$$
This simplified expression was also observed in \cite{LiS:2024}.
Note that
$$
|\widehat{\mathcal{Q}}_n| = (2n-1)!!,
$$
and one can check that
$$
|\mathcal{S}_n| = |\widehat{\mathcal{Q}}_n|\, \frac{2^{n-1}}{n!}.
$$
Even though there are less terms in the sum of \eqref{eq:B-spline-trig-weights-tris} compared to \eqref{eq:B-spline-trig-weights-full}, it is still more efficient to compute the expression in \eqref{eq:B-spline-trig-weights}; see Figure~\ref{fig:card-set} for the growth of $|\widehat{\mathcal{Q}}_n|$, compared with $|\mathcal{Q}_n|$ and $|\mathcal{S}_n|$.
Similarly, we have
$$
\hypb{w}_{j,m} = \dfrac{1}{|\widehat{\mathcal{Q}}_n|}
\sum_{\boldsymbol{q}\in \widehat{\mathcal{Q}}_n}\prod_{k=1}^{n}\cosh\left(\frac{x_{j+q_{2k}}-x_{j+q_{2k-1}}}{2}\right)
$$
for $m=2n+1\geq3$.
\end{remark}

\begin{remark}
The normalization weight $\trig{w}_{j,m}$ can also be expressed in integral form as follows:
\begin{equation}\label{eq:B-spline-trig-weights-int}
\trig{w}_{j,m} = \dfrac{2^{2n-1}}{|\mathcal{S}_n|}
\frac{1}{2\pi}\int_{0}^{2\pi}\prod_{k=1}^{2n}\sin\left(\dfrac{y-x_{j+k}}{2}\right)\textrm{d}y
\end{equation}
for $m=2n+1\geq3$. Its proof is delegated to Appendix~\ref{sec:B-spline-trig-weights-int}. However, due to the integration, the identity in \eqref{eq:B-spline-trig-weights-int} is less suited for practical computation.
There does not seem to be a hyperbolic version of this formula.
\end{remark}

Finally, let us consider the special case of uniform knots, i.e.,
\begin{equation}\label{eq:knots-uniform}
x_{j+k}=x_j+k h, \quad k=1,2,\ldots,m,
\end{equation}
for some real number $h>0$.
From the definition of the vector $\boldsymbol{s}$ in \eqref{eq:definition-s} it follows that its elements satisfy
$$
\sum_{k=1}^{2n-1}s_{k}=1.
$$
This gives
\begin{equation}\label{eq:y}
y_{j,m,\boldsymbol{s}}=-x_{j+1}+\sum_{k=1}^{2n-1}s_{k} x_{j+k+1}=\sum_{k=1}^{2n-1}s_{k} kh.
\end{equation}
It is clear that the value of $y_{j,m,\boldsymbol{s}}$ does not depend on the value of $j$, which implies that the normalization weights $\trig{w}_{j,m}$ and $\hypb{w}_{j,m}$ are constant with respect to $j$ in the case of uniform knots (see Theorem~\ref{thm:weights}).

To derive a simplified expression for the corresponding normalization weights, we make use of $q$-binomial coefficients (also called Gaussian coefficients). These coefficients can be obtained through recurrence.
We have
$$
\qbinom{b}{0}_q=\qbinom{a}{a}_q=1
$$
and
\begin{equation}\label{eq:q-binom}
\qbinom{a+b}{a}_q=\qbinom{a+b-1}{a-1}_q+q^a\qbinom{a+b-1}{a}_q
\end{equation}
for integers $a,b\geq1$. We refer the reader to, e.g., \cite[Section~1.6]{Aigner:2007} for details. If $q=1$ they become the standard binomial coefficients, i.e.,
$$
\qbinom{a+b}{a}_1=\binom{a+b}{a}.
$$
In general, $q$-binomial coefficients are a polynomial in $q$, which can be expressed as
\begin{equation}\label{eq:rho}
\qbinom{a+b}{a}_q=\sum_{i=0}^{ab} \rho_{i,a,b}\, q^i.
\end{equation}
The value $\rho_{i,a,b}$ can be interpreted as the number of ways of throwing $i$ indistinguishable balls into $a$ indistinguishable bins, where each bin can contain up to $b$ balls. All the $\rho_{i,a,b}$, $i=0,\ldots,ab$, appearing in the sum of \eqref{eq:rho}, can be computed using the method described in Algorithm~\ref{alg:rho}, which is based on \eqref{eq:q-binom}.

\begin{algorithm}[t!]
\caption{Computation of the vector $\boldsymbol{\rho}_{a,b}$ containing the values $\rho_{i,a,b}$, $i=0,\ldots,ab$, appearing in the sum of \eqref{eq:rho}, for given nonnegative integers $a$ and $b$.
Let $\boldsymbol{1}_k$ be the vector of length $k$ consisting of all ones.}\label{alg:rho}
\medskip
\Fn{$\operatorname{rho}(a,b)$}{
$\boldsymbol{\rho}_{a,b} \gets \boldsymbol{1}_{ab+1}$\;
\If{$a \geq 2$ \And $b \geq 2$}{
    $\boldsymbol{\rho}_{a-1,b} \gets \operatorname{rho}(a-1,b)$\;
    $\boldsymbol{\rho}_{a,b-1} \gets \operatorname{rho}(a,b-1)$\;
    \For{$i=0,\ldots,a-1$}{
        $\rho_{i,a,b} \gets \rho_{i,a-1,b}$\;
    }
    \For{$i=a,\ldots,(a-1)b$}{
        $\rho_{i,a,b} \gets \rho_{i,a-1,b} + \rho_{i-a,a,b-1}$\;
    }
    \For{$i=(a-1)b+1,\ldots,ab$}{
        $\rho_{i,a,b} \gets \rho_{i-a,a,b-1}$\;
    }
}
\Return $\boldsymbol{\rho}_{a,b}$\;
}
\end{algorithm}

\begin{theorem}\label{thm:weights-uniform}
Let $m=2n+1\geq3$ and consider the case of uniform knots as in \eqref{eq:knots-uniform}.
The normalization weight in \eqref{eq:B-spline-trig-weights-full} can be expressed as
\begin{equation}\label{eq:B-spline-trig-weights-uniform}
\trig{w}_{j,m}=\dfrac{1}{|\mathcal{S}_n|}
\sum_{i=0}^{n(n-1)}\rho_{i,n-1,n}\cos\left(\frac{n^2-2i}{2}h\right),
\end{equation}
and similarly, the normalization weight in \eqref{eq:B-spline-hypb-weights-full} can be expressed as
\begin{equation}\label{eq:B-spline-hypb-weights-uniform}
\hypb{w}_{j,m}=\dfrac{1}{|\mathcal{S}_n|}
\sum_{i=0}^{n(n-1)}\rho_{i,n-1,n}\cosh\left(\frac{n^2-2i}{2}h\right).
\end{equation}
The value $\rho_{i,a,b}$ is defined in \eqref{eq:rho}.
\end{theorem}
\begin{proof}
We only consider the trigonometric case (the proof of the hyperbolic case is analogous).
Let us go back to the expression of $y_{j,m,\boldsymbol{s}}$ in \eqref{eq:y}. After factoring out $h$, we can split the last sum in two parts:
$$
\sum_{k=1}^{2n-1}s_{k} k=\sigma_+ + \sigma_-,
$$
where $\sigma_+$ collects the positive terms while $\sigma_-$ the negative terms, and thus
\begin{equation}\label{eq:sigma-diff}
\binom{2n}{2} = \sum_{k=1}^{2n-1} k = \sigma_+ - \sigma_-.
\end{equation}
The number of vectors $\boldsymbol{s}\in \mathcal{S}_n$ with the same value for $\sigma_-$ (and also the same value for $\sigma_+$) equals
$$
\rho_{-\sigma_--\binom{n}{2},n-1,n} =
\rho_{\sigma_+-\binom{n+1}{2},n,n-1},
$$
with $\rho_{i,a,b}$ defined in \eqref{eq:rho}.

Looking at \eqref{eq:B-spline-trig-weights}, we can rewrite each cosine term as
$$
c_{j,m,\boldsymbol{s}}
= \cos\left(\frac{y_{j,m,\boldsymbol{s}}}{2}\right)
= \cos\left(\frac{\sigma_+ + \sigma_-}{2}h\right).
$$
Then, using \eqref{eq:sigma-diff} and the variable $i=-\sigma_--\binom{n}{2}$, we get
$$
c_{j,m,\boldsymbol{s}}
= \cos\left(\frac{\binom{2n}{2} + 2\sigma_-}{2}h\right)
= \cos\left(\frac{n^2-2i}{2}h\right).
$$
This completes the proof of \eqref{eq:B-spline-trig-weights-uniform}.
\end{proof}

Note that
$$
\sum_{i=0}^{n(n-1)} \rho_{i,n-1,n}=\qbinom{2n-1}{n-1}_1=\binom{2n-1}{n-1}=|\mathcal{S}_n|.
$$
Furthermore, we have
\begin{equation}\label{eq:symmetry-uniform}
n^2-2i_1 = -(n^2-2i_2) \quad \text{if\ } i_1+i_2=n^2.
\end{equation}
Since the functions $\cos$ and $\cosh$ are even, one could merge the terms in
\eqref{eq:B-spline-trig-weights-uniform} and \eqref{eq:B-spline-hypb-weights-uniform} related to \eqref{eq:symmetry-uniform} for the range $n\leq i_1<n^2/2$ (which implies $n^2/2<i_2\leq n(n-1)$).

\begin{remark}
Considering the uniform setting as in \eqref{eq:knots-uniform}, the following alternative expression for $\trig{w}_{j,m}$ can be found in \cite{Walz:1997}:
$$
\begin{aligned}
\trig{w}_{j,m}=
\frac{2^{n-1}}{|\mathcal{S}_n|}
\sum_{i=0}^{\lfloor n/2 \rfloor} \frac{1}{2^{2i}}\binom{2i}{i}\sum_{\substack{1\leq q_1<q_2<\cdots \\ \hfill <q_{n-2i}\leq n}} \prod_{k=1}^{n-2i}\cos\left(\frac{2q_k-1}{2}h\right).
\end{aligned}
$$
The number of terms in this expression equals
$$
\sum_{i=0}^{\lfloor n/2 \rfloor}\binom{n}{n-2i}
= 2^{n-1},
$$
which clearly grows faster in $n$ than the number of terms needed in \eqref{eq:B-spline-trig-weights-uniform}, an exponential versus polynomial growth.
\end{remark}

\section{Normalized B-spline examples}\label{sec:examples}

The normalized B-splines $\trig{N}_{j,m}$ and $\hypb{N}_{j,m}$ can be computed through their definitions in \eqref{eq:B-spline-trig-def} and \eqref{eq:B-spline-hypb-def}, respectively. Alternatively, the recurrence relations for $T_{j,m}$ and $H_{j,m}$ can be rephrased directly in terms of normalized B-splines as follows. For $m=1$ it is clear that
$$
\trig{N}_{j,1}(x)=\hypb{N}_{j,1}(x)=\begin{cases}
1 & \text{if } x_j\leq x< x_{j+1},\\
0 & \text{otherwise}.
\end{cases}
$$
Since the normalized B-splines are only available for odd order, we arrive at a three-term recurrence. Let us introduce the shorthand notation
$$
\trig{s}(x,y)=\sin\left(\dfrac{x-y}{2}\right), \quad
\hypb{s}(x,y)=\sinh\left(\dfrac{x-y}{2}\right).
$$
Then, for $m=2n+1\geq3$ and $X\in\{T,H\}$,
\begin{equation}\label{eq:recursion-normalized}
\begin{aligned}
\gene{N}_{j,m}(x)
&= \dfrac{\gene{s}(x,x_j)\gene{s}(x,x_j)}{\gene{s}(x_{j+m-1},x_j)\gene{s}(x_{j+m-2},x_j)}\dfrac{\gene{w}_{j,m}}{\gene{w}_{j,m-2}}\gene{N}_{j,m-2}(x)\\
&+\biggl[\dfrac{\gene{s}(x,x_j)\gene{s}(x_{j+m-1},x)}{\gene{s}(x_{j+m-1},x_j)\gene{s}(x_{j+m-1},x_{j+1})}\\
&\quad+\dfrac{\gene{s}(x_{j+m},x)\gene{s}(x,x_{j+1})}{\gene{s}(x_{j+m},x_{j+1})\gene{s}(x_{j+m-1},x_{j+1})}\biggr]\dfrac{\gene{w}_{j,m}}{\gene{w}_{j+1,m-2}}\gene{N}_{j+1,m-2}(x)\\
&+\dfrac{\gene{s}(x_{j+m},x)\gene{s}(x_{j+m},x)}{\gene{s}(x_{j+m},x_{j+1})\gene{s}(x_{j+m},x_{j+2})}\dfrac{\gene{w}_{j,m}}{\gene{w}_{j+2,m-2}}\gene{N}_{j+2,m-2}(x).
\end{aligned}
\end{equation}

\begin{figure}[t!]
\centering
\subfigure[B-spline basis of order $m=3$]{
\includegraphics[height=5cm]{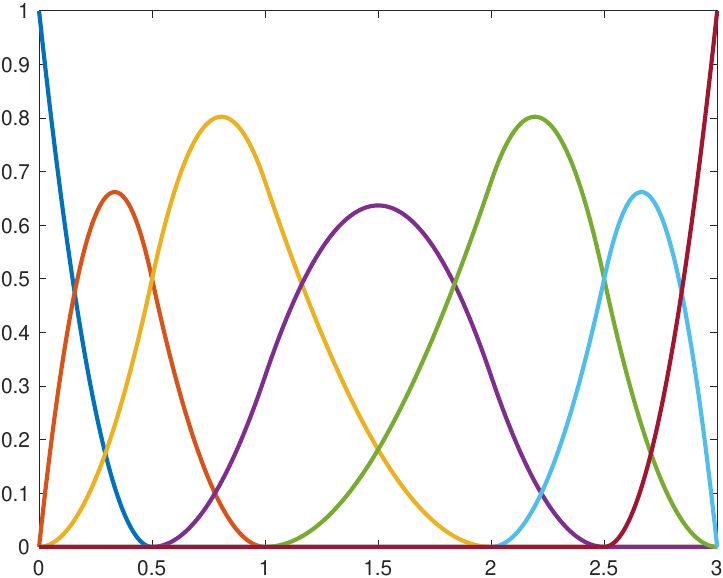}} \hspace*{0.5cm}
\subfigure[B-spline basis of order $m=5$]{
\includegraphics[height=5cm]{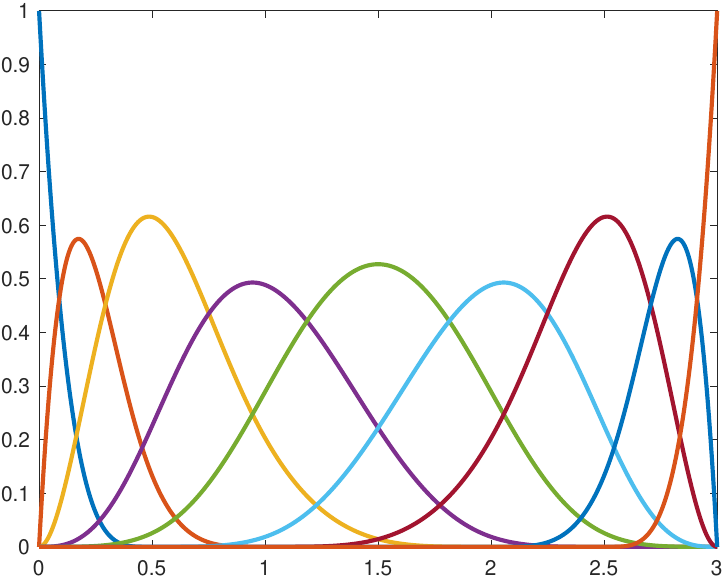}} \\
\subfigure[B-spline basis of order $m=7$]{
\includegraphics[height=5cm]{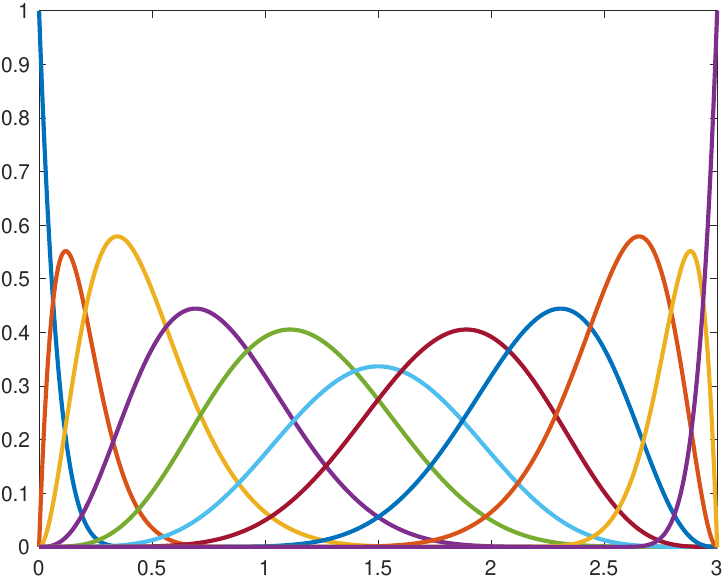}} \hspace*{0.5cm}
\subfigure[B-spline basis of order $m=9$]{
\includegraphics[height=5cm]{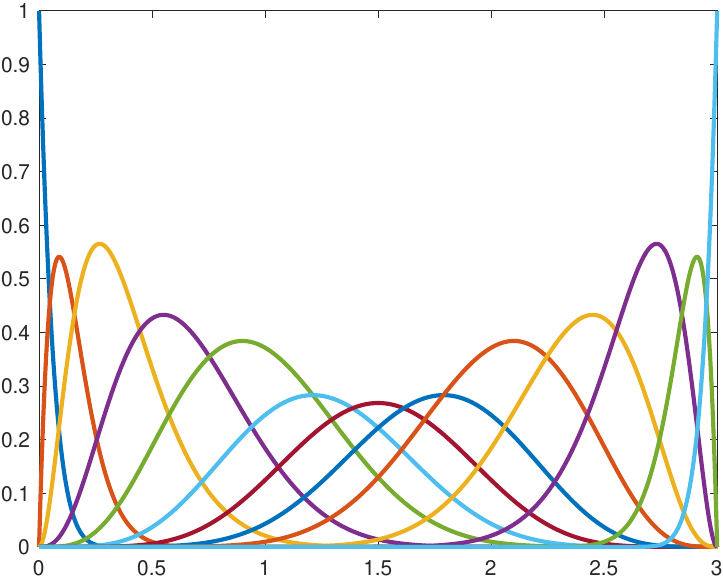}}
\caption{Visualization of sets of normalized trigonometric B-splines $\trig{N}_{j,m}$ defined on the knot sequence \eqref{eq:ex-B-splines-knots} for different values of $m$ as described in Example~\ref{ex:B-splines}. The use of open knots leads to interpolation at the end points of the interval.}\label{fig:B-splines-trig}
\end{figure}
\begin{figure}[t!]
\centering
\subfigure[B-spline basis of order $m=3$]{
\includegraphics[height=5cm]{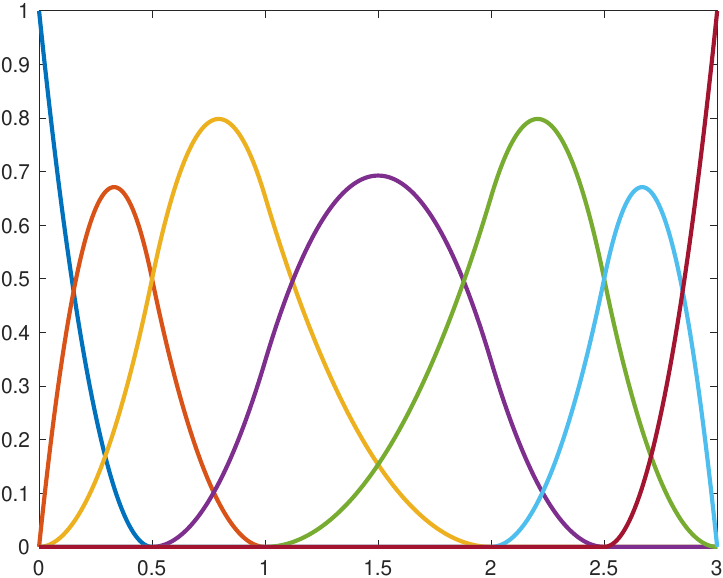}} \hspace*{0.5cm}
\subfigure[B-spline basis of order $m=5$]{
\includegraphics[height=5cm]{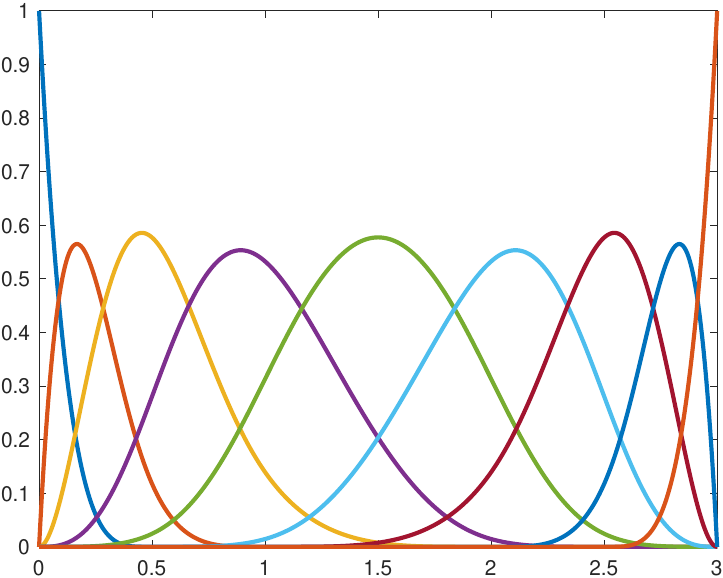}} \\
\subfigure[B-spline basis of order $m=7$]{
\includegraphics[height=5cm]{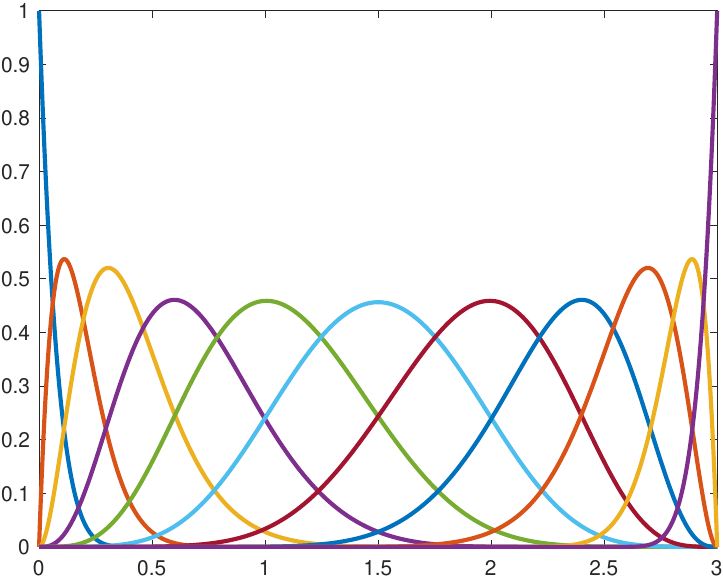}} \hspace*{0.5cm}
\subfigure[B-spline basis of order $m=9$]{
\includegraphics[height=5cm]{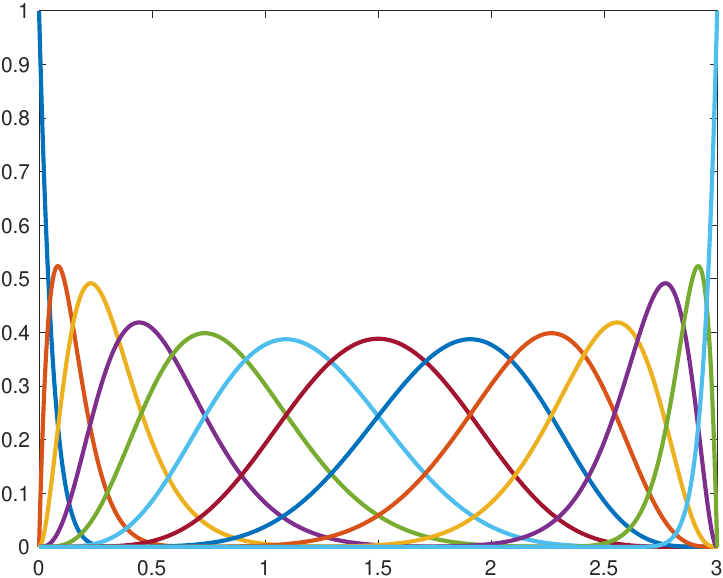}}
\caption{Visualization of sets of normalized hyperbolic B-splines $\hypb{N}_{j,m}$ defined on the knot sequence \eqref{eq:ex-B-splines-knots} for different values of $m$ as described in Example~\ref{ex:B-splines}. The use of open knots leads to interpolation at the end points of the interval.}\label{fig:B-splines-hypb}
\end{figure}

Like for standard (polynomial) normalized B-splines, choosing an open knot configuration $\{x_j\}$, i.e., the first $m$ and the last $m$ knots coincide, leads to interpolation at the end points of the interval. This is illustrated in the following example.

\begin{example}\label{ex:B-splines}
Consider the knot sequence
\begin{equation}\label{eq:ex-B-splines-knots}
\biggl\{\underbrace{0,\ldots,0}_{m \text{ times}},\frac{1}{2},1,2,\frac{5}{2},\underbrace{3,\ldots,3}_{m \text{ times}}\biggr\}
\end{equation}
for different values of $m=2n+1$ ($n=1,2,3,4$). The corresponding normalized B-splines $\trig{N}_{j,m}$ and $\hypb{N}_{j,m}$ are visualized in Figures~\ref{fig:B-splines-trig} and~\ref{fig:B-splines-hypb}, respectively. Note that the dimension of the spline space over the knot sequence \eqref{eq:ex-B-splines-knots} equals $m+4$.
\end{example}

Let $m$ odd, $m\leq l\leq r$, and $x_l<x_{r+1}$. A normalized trigonometric or hyperbolic B-spline curve $\boldsymbol{C}$ is given by
$$
\boldsymbol{C}(x)=\sum_{j=l+1-m}^{r}\boldsymbol{P}_j\gene{N}_{j,m}(x),
\quad x_l\leq x<x_{r+1},
$$
with $X\in\{T,H\}$ and some set of control points $\{\boldsymbol{P}_j\}$.
By the partition of unity and the nonnegativity of the B-splines, such representation is particularly suited for geometric modeling.
Moreover, conic sections can be easily described by means of normalized trigonometric and hyperbolic B-spline curves.
In the following examples we show some simple representations of a circle.

\begin{example}\label{ex:circle}
A normalized trigonometric B-spline curve of order $m=3$ describing a full circle (with radius $r=1$) can be obtained by considering a uniform knot sequence of the form
\begin{equation}\label{eq:ex-circle-knots}
\biggl\{-\frac{4\pi}{p},-\frac{2\pi}{p},0,\frac{2\pi}{p},\ldots,\frac{2(p-1)\pi}{p},2\pi,\frac{2(p+1)\pi}{p},\frac{2(p+2)\pi}{p}\biggr\}
\end{equation}
and by selecting the corresponding control points as the corners of a regular $p$-sided polygon 
($p\geq3$ integer),
\begin{equation}\label{eq:ex-circle-points}
\boldsymbol{P}_j = \left(\frac{\cos(\theta+2j\pi/p)}{\cos(\pi/p)}, \frac{\sin(\theta+2j\pi/p)}{\cos(\pi/p)}\right)^T
\end{equation}
for $j=1,\ldots,p+2$, with any real value $\theta$.
The parametric domain is $[0,2\pi)$ and the first two control points coincide with the last two control points. It is easy to check that $\trig{w}_{j,3}=\cos(\pi/p)$.
The cases $p=4,8$ are illustrated in Figure~\ref{fig:circle}. Note that the case $p=4$ does not satisfy the requirement $x_{j+3}-x_j<\pi$ (Remark~\ref{rmk:knots-trig}), but it still leads to a valid set of normalized trigonometric B-splines; see also the limit case $\ell=0$ in \cite[Example~8]{Speleers:2022}. Finally, we highlight that the spline representation is globally $C^1$ and each leg of the $p$-sided polygon is tangent to the circle, which is a nice feature for design.
\end{example}

\begin{figure}[t!]
\centering
\subfigure[Control points and B-spline basis for $m=3$, $p=4$]
{\includegraphics[height=5cm]{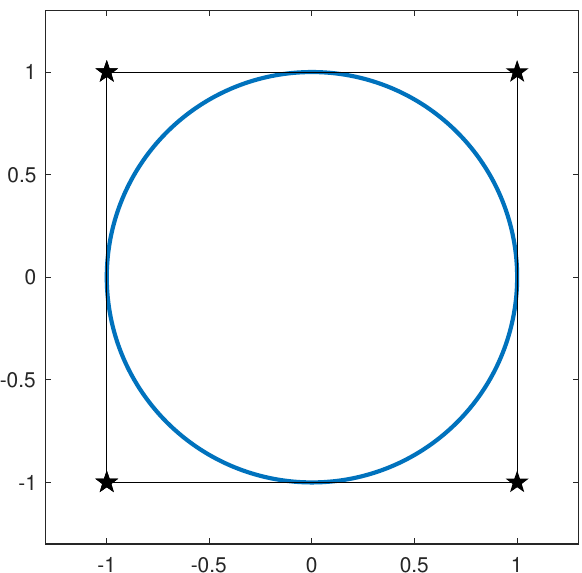} \hspace*{0.8cm}
 \includegraphics[height=5cm]{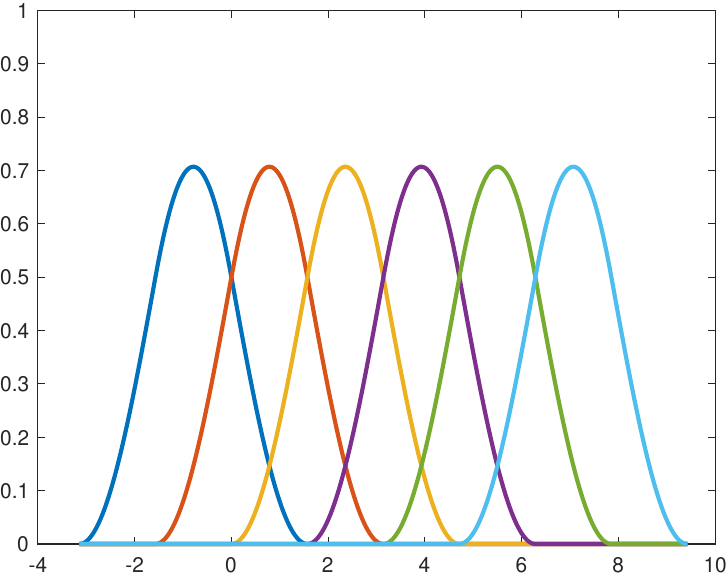}} \\
\subfigure[Control points and B-spline basis for $m=3$, $p=8$]
{\includegraphics[height=5cm]{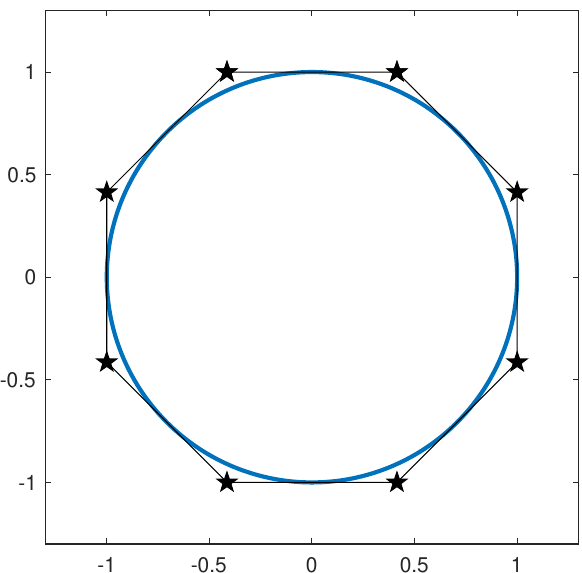} \hspace*{0.8cm}
 \includegraphics[height=5cm]{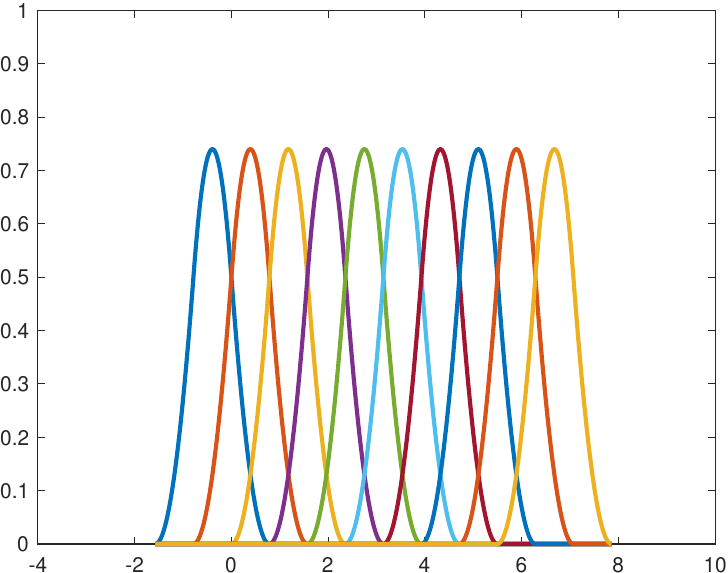}}
\caption{Representation of a full circle by means of a normalized trigonometric B-spline curve of order $m=3$ defined on the knot sequence \eqref{eq:ex-circle-knots} for $p=4,8$ as described in Example~\ref{ex:circle} ($\theta=\pi/p$). The corresponding control points are indicated with the symbol~$\star$ and form a regular $p$-sided polygon 
that is tangential to the circle.}\label{fig:circle}
\end{figure}

A similar representation of a full circle can be formulated for any odd order.
\begin{example}\label{ex:circle-m}
Let $m=2n+1\geq3$ and consider a uniform knot sequence of the form
\begin{equation}\label{eq:ex-circle-knots-m}
\biggl\{\frac{2k\pi}{p}:k=-2n,-2n+1,\ldots,p+2n \biggr\}
\end{equation}
for some integer $p\geq m$.
Then, selecting the same control points as in
\eqref{eq:ex-circle-points} for $j=1,\ldots,p+2n$, with any real value~$\theta$, results in a normalized trigonometric B-spline curve of order $m$ describing a full circle. The parametric domain is again $[0,2\pi)$. The cases $m=5,7$ are illustrated in Figure~\ref{fig:circle-m}, where $p=8$. We observe that for $m>3$ the $p$-sided polygon is not tangent to the circle anymore.
Since all the knots in \eqref{eq:ex-circle-knots-m} are distinct, the spline representation has maximal smoothness and is globally $C^{m-2}$. On the other hand, a standard NURBS representation of a full circle with the same smoothness has to be at least of order $2m-1$ instead of $m$ \cite{BangertP:1997}.
\end{example}

\begin{figure}[t!]
\centering
\subfigure[Control points and B-spline basis for $m=5$, $p=8$]
{\includegraphics[height=5cm]{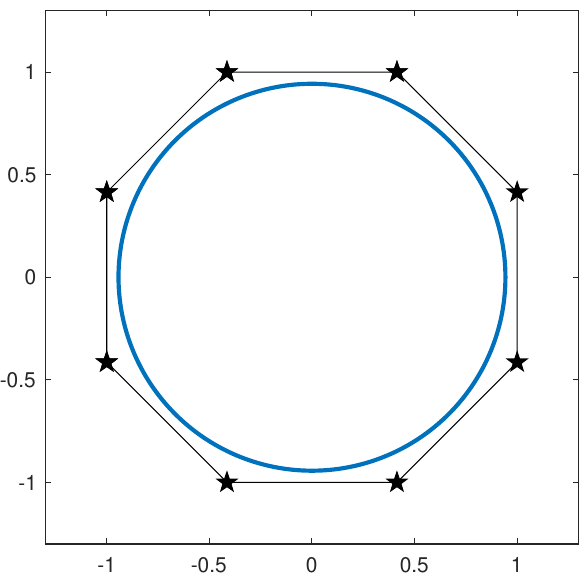} \hspace*{0.8cm}
 \includegraphics[height=5cm]{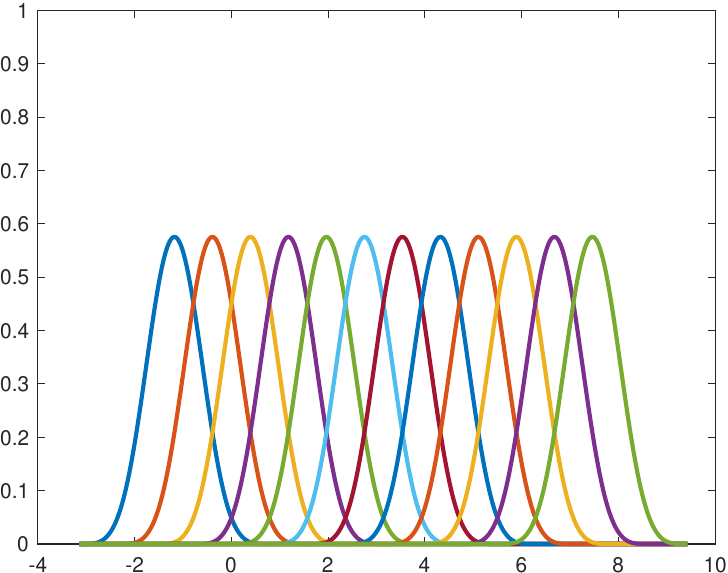}} \\
\subfigure[Control points and B-spline basis for $m=7$, $p=8$]
{\includegraphics[height=5cm]{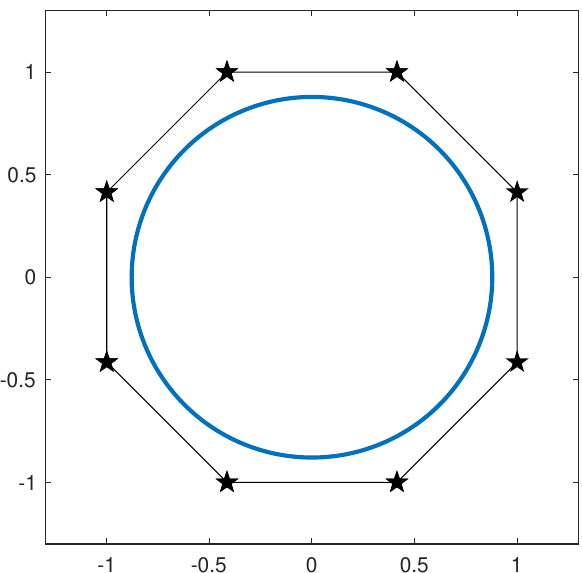} \hspace*{0.8cm}
 \includegraphics[height=5cm]{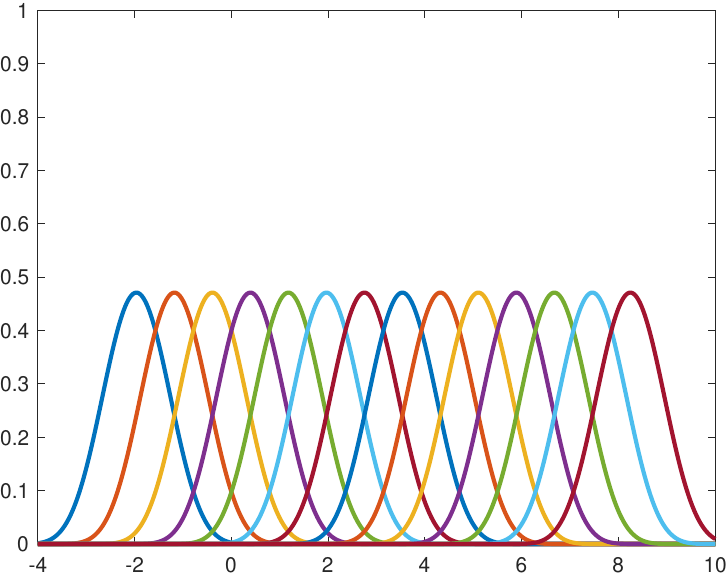}}
\caption{Representation of a full circle by means of a normalized trigonometric B-spline curve of order $m=5,7$ defined on the knot sequence \eqref{eq:ex-circle-knots-m} for $p=8$ as described in Example~\ref{ex:circle-m} ($\theta=\pi/p$). The corresponding control points are indicated with the symbol~$\star$ and form a regular $p$-sided polygon.}\label{fig:circle-m}
\end{figure}

Next, we illustrate a description of a circle segment that is obtained from the previous uniform representations of the full circle by inserting knots at the endpoints of a certain interval in the parameter domain.
\begin{example}\label{ex:circle-part-m}
Let $m=2n+1\geq3$ and consider the representation of a full circle elaborated in Example~\ref{ex:circle-m} over the parametric domain $[0,2\pi)$.
To obtain a circle segment, we restrict the parametric domain to the interval $[2a\pi/p,2b\pi/p)$ for some integers $0 \leq a<b \leq p$ and take the open knot sequence
\begin{equation}\label{eq:ex-circle-part-knots-m}
\biggl\{\frac{2k\pi}{p}:k=\underbrace{a,\ldots,a}_{m \text{ times}},a+1,\ldots,b-1,\underbrace{b,\ldots,b}_{m \text{ times}} \biggr\}.
\end{equation}
The corresponding representation in terms of normalized trigonometric B-splines of order $m=5,7$ is illustrated in Figure~\ref{fig:circle-part-m}, where $p=8$, $a=3-n$, and $b=p+1-n$ (three quarters of a circle). The same parameterization is maintained as in Example~\ref{ex:circle-m} (Figure~\ref{fig:circle-m}), which can be computed by means of knot insertion \cite{KochLNS:1995}. The expressions of the control points are reported in Figure~\ref{tab:circle-part-m}.
\end{example}

\begin{figure}[t!]
\centering
\subfigure[Control points and B-spline basis for $m=5$, $p=8$]
{\includegraphics[height=5cm]{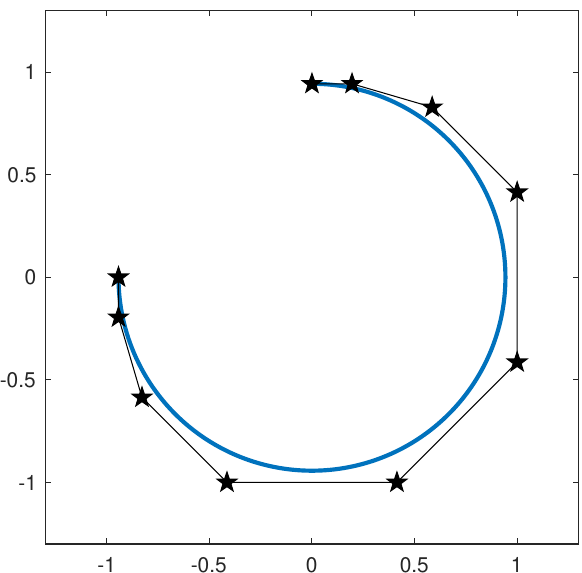} \hspace*{0.8cm}
 \includegraphics[height=5cm]{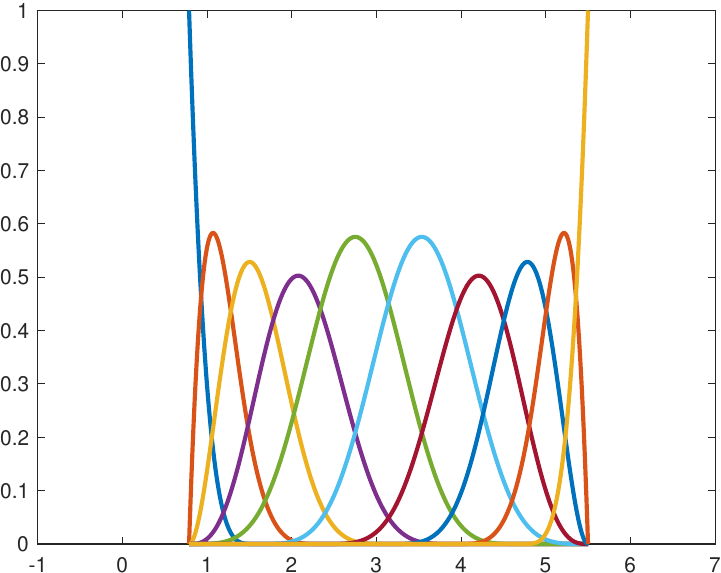}} \\
\subfigure[Control points and B-spline basis for $m=7$, $p=8$]
{\includegraphics[height=5cm]{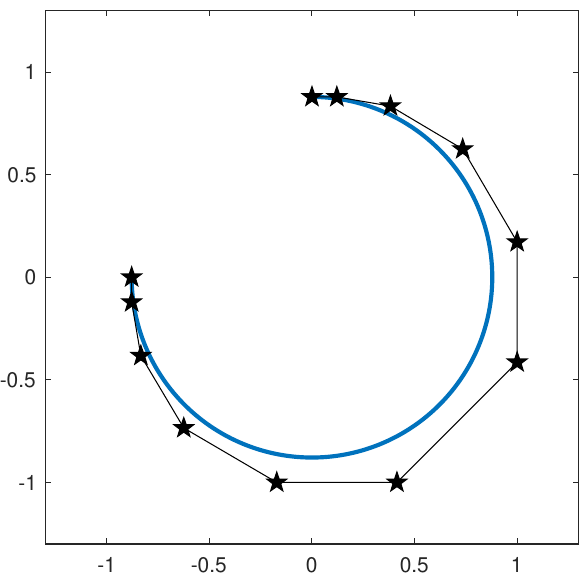} \hspace*{0.8cm}
 \includegraphics[height=5cm]{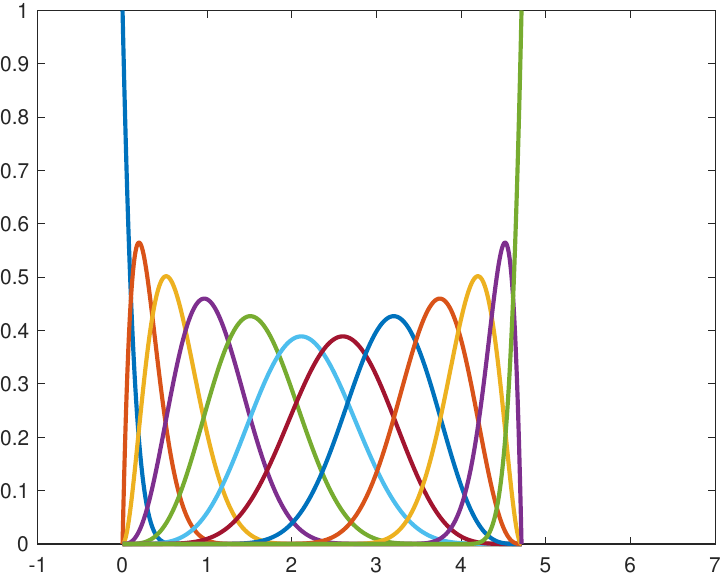}}
\caption{Representation of three quarters of a circle by means of a normalized trigonometric B-spline curve of order $m=5,7$ ($n=2,3$) defined on the knot sequence \eqref{eq:ex-circle-part-knots-m} for $p=8$, $a=3-n$, and $b=p+1-n$ as described in Example~\ref{ex:circle-part-m} ($\theta=\pi/p$). The corresponding control points are indicated with the symbol~$\star$.}\label{fig:circle-part-m}
\end{figure}

\begin{figure}[t!]
\makegapedcells
\begin{tabularx}{\textwidth}{|>{\centering\arraybackslash}X|}
\hline
$m=5$ \\
\hline
$\footnotesize
\begin{gathered}
\boldsymbol{P}_1 = \left( -\frac{2\sqrt{2}}{3},\, 0 \right)^T, \quad
\boldsymbol{P}_2 = \left( -\frac{2\sqrt{2}}{3},\, -\frac{2}{3} + \frac{\sqrt{2}}{3} \right)^T, \quad
\boldsymbol{P}_3 = \left( 2 - 2\sqrt{2},\, -2 + \sqrt{2} \right)^T, \\
\boldsymbol{P}_4 = \left( 1 -\sqrt{2},\, -1 \right)^T, \quad
\boldsymbol{P}_5 = \left( -1 + \sqrt{2},\, -1 \right)^T, \quad
\boldsymbol{P}_6 = \left( 1,\, 1 - \sqrt{2} \right)^T, \quad
\boldsymbol{P}_7 = \left( 1,\, -1 + \sqrt{2} \right)^T, \\
\boldsymbol{P}_8 = \left( 2 - \sqrt{2},\, -2 + 2\sqrt{2} \right)^T, \quad
\boldsymbol{P}_9 = \left( \frac{2}{3} - \frac{\sqrt{2}}{3},\, \frac{2\sqrt{2}}{3} \right)^T, \quad
\boldsymbol{P}_{10} = \left( 0,\, \frac{2\sqrt{2}}{3} \right)^T
\end{gathered}
$ \\
\hline
\end{tabularx}

\smallskip

\begin{tabularx}{\textwidth}{|>{\centering\arraybackslash}X|}
\hline
$m=7$ \\
\hline
$\footnotesize
\begin{gathered}
\boldsymbol{P}_1 = \left( -3 + \frac{3\sqrt{2}}{2},\, 0 \right)^T, \quad
\boldsymbol{P}_2 = \left( -3 + \frac{3\sqrt{2}}{2},\, 2 - \frac{3\sqrt{2}}{2} \right)^T, \quad
\boldsymbol{P}_3 = \left( -\frac{32}{7} + \frac{37\sqrt{2}}{14},\, \frac{15}{7} - \frac{25\sqrt{2}}{14} \right)^T, \\
\boldsymbol{P}_4 = \left( -\frac{27}{7} + \frac{16\sqrt{2}}{7},\, \frac{9}{7} - \frac{10\sqrt{2}}{7} \right)^T, \quad
\boldsymbol{P}_5 = \left( -3 + 2\sqrt{2},\, -1 \right)^T, \quad
\boldsymbol{P}_6 = \left( -1 + \sqrt{2},\, -1 \right)^T, \\
\boldsymbol{P}_7 = \left( 1,\, 1 - \sqrt{2} \right)^T, \quad
\boldsymbol{P}_8 = \left( 1,\, 3 - 2\sqrt{2} \right)^T, \quad
\boldsymbol{P}_9 = \left( -\frac{9}{7} + \frac{10\sqrt{2}}{7},\, \frac{27}{7} - \frac{16\sqrt{2}}{7} \right)^T, \\
\boldsymbol{P}_{10} = \left( -\frac{15}{7} + \frac{25\sqrt{2}}{14},\, \frac{32}{7} - \frac{37\sqrt{2}}{14} \right)^T, \quad
\boldsymbol{P}_{11} = \left( -2 + \frac{3\sqrt{2}}{2},\, 3 - \frac{3\sqrt{2}}{2} \right)^T, \quad
\boldsymbol{P}_{12} = \left( 0,\, 3 - \frac{3\sqrt{2}}{2} \right)^T
\end{gathered}
$ \\
\hline
\end{tabularx}
\caption{Expressions of the control points used in Figure~\ref{fig:circle-part-m} for the representation of three quarters of a circle by means of a normalized trigonometric B-spline curve of order $m=5,7$ ($n=2,3$) defined on the knot sequence \eqref{eq:ex-circle-part-knots-m} for $p=8$, $a=3-n$, and $b=p+1-n$ as described in Example~\ref{ex:circle-part-m} ($\theta=\pi/p$).}\label{tab:circle-part-m}
\end{figure}

We end this section by evincing that working with high-order B-splines of maximal smoothness is beneficial for approximation purposes and in particular for isogeometric analysis (where both design and approximation is of interest).
\begin{example}\label{ex:approx}
We want to build a spline approximation of the oscillatory target function
\begin{equation}\label{eq:ex-approx-fun}
f(x)=\sin(10x)\, \frac{4(x/5-1)^2+1}{5}, \quad x\in[0,10].
\end{equation}
This function is visualized in Figure~\ref{fig:ex-approx-fun}.
Let $m=2n+1\geq3$ and consider a knot sequence of the form
\begin{equation}\label{eq:ex-approx-knots}
\biggl\{\frac{k}{p}:k=\underbrace{0,\ldots,0}_{m \text{ times}},1,\ldots,10p-1,\underbrace{10p,\ldots,10p}_{m \text{ times}} \biggr\},
\end{equation}
where $p=2^{l+1}$ and $l=1,\ldots,5$. Then, we look for a trigonometric and hyperbolic B-spline representation of order $m$ on this knot sequence. The number of degrees of freedom (NDOF) is $10p+m-1$. The representation is obtained by applying the least-squares method to approximate $f$ sampled at $10001$ equally distributed points in $[0,10]$. The corresponding errors measured in the $L_\infty$ norm are plotted against the NDOF in Figure~\ref{fig:ex-approx-error}. The approximation error is expected to behave like $\mathcal{O}(p^{-m})$ and hence the convergence rate is faster for higher order, both in the trigonometric and hyperbolic case. From Figure~\ref{fig:ex-approx-error} we may conclude that splines of higher order (and maximal smoothness) possess a higher accuracy per degree of freedom.
\end{example}

\begin{figure}[t!]
\centering
\includegraphics[height=5cm]{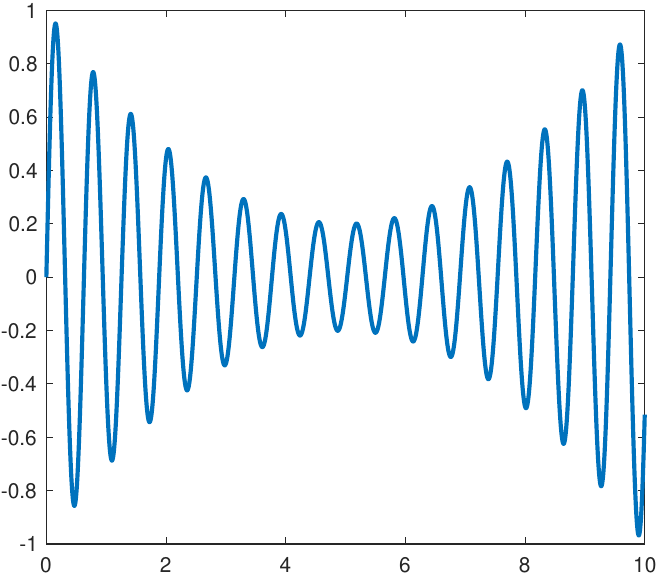}
\caption{Target function $f$ defined in \eqref{eq:ex-approx-fun}. }\label{fig:ex-approx-fun}
\bigskip
\centering
\subfigure[Trigonometric spline errors]
{\includegraphics[height=5cm]{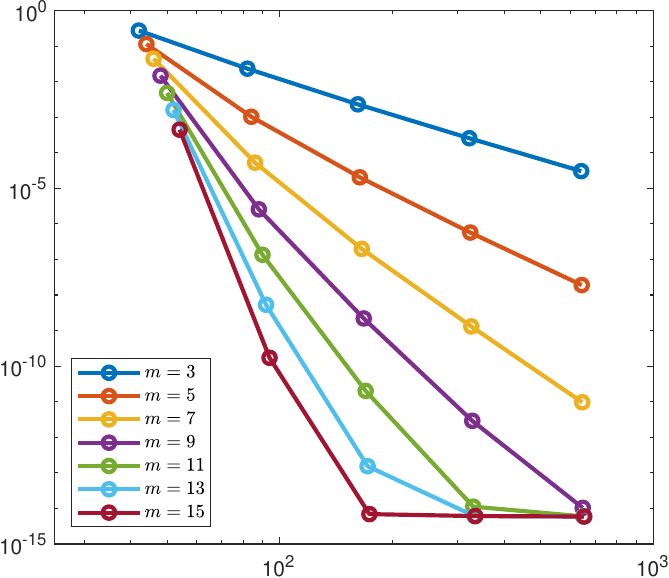}} \hspace*{0.8cm}
\subfigure[Hyperbolic spline errors]
{\includegraphics[height=5cm]{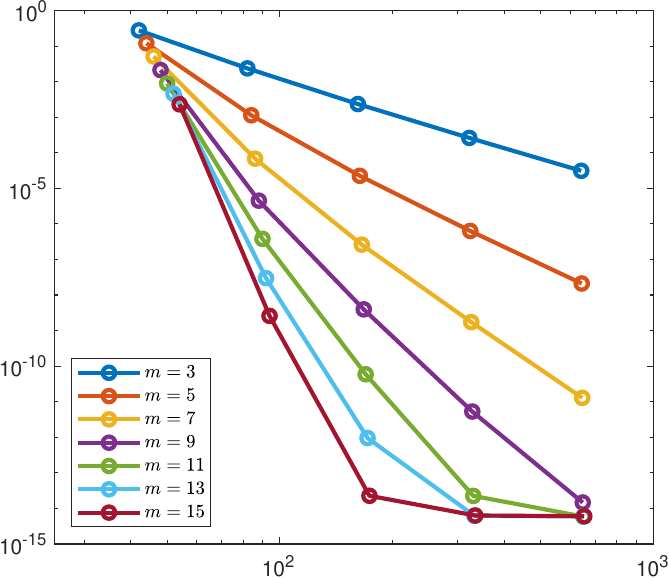}}
\caption{Convergence of the $L_{\infty}$ error with respect to the NDOF of least-squares approximations in the trigonometric and hyperbolic spline spaces of order $m=2n+1$, $n=1,\ldots,7$ defined on the knot sequence \eqref{eq:ex-approx-knots} for $p=2^{l+1}$, $l=1,\ldots,5$, in logarithmic scale, as described in Example~\ref{ex:approx}. }\label{fig:ex-approx-error}
\end{figure}

\section{Concluding remarks}\label{sec:conclusion}

Trigonometric and hyperbolic splines are important variations of classical polynomial splines because they are well suited for building high-quality parametric descriptions of conic sections and cam profiles. Moreover, they can be represented in terms of B-spline bases, which can be computed efficiently and accurately via recurrence relations analogous to the classical polynomial B-spline basis; see \eqref{eq:recursion-trig} and~\eqref{eq:recursion-hypb}.

However, in their original formulation, these two types of B-splines do not form a partition of unity and consequently do not admit the notion of control polygons with the convex hull property for design purposes. In the literature one can find the explicit formulas \eqref{eq:B-spline-trig-weights-full} and \eqref{eq:B-spline-hypb-weights-full} for the associated normalization weights, but they are computationally expensive; the growth of $|\mathcal{Q}_n|$ is depicted in Figure~\ref{fig:card-set}.

We have derived the new explicit expressions \eqref{eq:B-spline-trig-weights} and \eqref{eq:B-spline-hypb-weights} for such weights, requiring less arithmetic operations as indicated by the growth of $|\mathcal{S}_n|$ in Figure~\ref{fig:card-set}. We have also provided a practical method to compute them recursively; see Algorithm~\ref{alg:weights}. In the special case of uniform knots, we have further simplified the expressions into \eqref{eq:B-spline-trig-weights-uniform} and \eqref{eq:B-spline-hypb-weights-uniform}, which rely on $q$-binomial coefficients.

These weights can then be used in the construction of normalized trigonometric/hyperbolic B-splines. As example application, we have considered the exact representation of a circle in terms of $C^{2n-1}$ trigonometric B-splines of order $m=2n+1\geq3$, with a variable number of control points; see Examples~\ref{ex:circle} to~\ref{ex:circle-part-m}.
We have also illustrated the approximation power of trigonometric and hyperbolic splines. High-order splines (of maximal smoothness) exhibit an excellent accuracy per degree of freedom; see Example~\ref{ex:approx}.

It might be opportune to incorporate the recurrence formula \eqref{eq:recursion-normalized} into the Matlab toolbox for MDTB-splines in \cite{Speleers:2022} to achieve a more stable computation in the case of trigonometric/hyperbolic B-splines (of uniform degree).

\section*{Acknowledgements}

This work has been partially supported by the MUR Excellence Department Project MatMod@TOV (CUP E83C23000330006) awarded to the Department of Mathematics of the University of Rome Tor Vergata
and by the Italian Research Center on High Performance Computing, Big Data and Quantum Computing (CUP E83C22003230001).
H.~Speleers is a member of the research group GNCS (Gruppo Nazionale per il Calcolo Scientifico) of INdAM (Istituto Nazionale di Alta Matematica).

\appendix
\section{Proof of the identity \eqref{eq:B-spline-trig-weights-int}.}
\label{sec:B-spline-trig-weights-int}

In this appendix, we prove the identity \eqref{eq:B-spline-trig-weights-int}. The starting point is the trigonometric analog of Marsden's identity derived in \cite{LycheW:1979}.
Let $m=2n+1\geq3$, $m\leq l\leq r$, and $x_l<x_{r+1}$. It holds for any $x\in[x_l,x_{r+1})$ and $y\in\mathbb{R}$,
\begin{equation}\label{eq:B-spline-trig-marsden}
\left[\sin\left(\dfrac{y-x}{2}\right)\right]^{2n}=\sum_{j=l+1-m}^{r}\phi_{j,m}(y)\sin\left(\dfrac{x_{j+m}-x_j}{2}\right) T_{j,m}(x),
\end{equation}
where
$$
\phi_{j,m}(y)=\prod_{k=1}^{2n}\sin\left(\dfrac{y-x_{j+k}}{2}\right).
$$
Let $\mathcal{L}_y(f)$ be the operator that selects the constant term in the trigonometric polynomial expansion of the function $f$ in the variable $y$.
We now compare the constant terms in such expansions of both sides of the identity \eqref{eq:B-spline-trig-marsden}. For the left-hand side, this term is known in explicit form (see, e.g., \cite{Walz:1997}),
$$
\mathcal{L}_y\left(\left[\sin\left(\dfrac{y-x}{2}\right)\right]^{2n}\right)=\frac{1}{2^{2n}}\binom{2n}{n}=\frac{|\mathcal{S}_n|}{2^{2n-1}}.
$$
For the right-hand side, we utilize the standard integral formula to compute the Fourier coefficient related to the constant term in the expansion of $\phi_{j,m}$,
$$
\mathcal{L}_y\left(\phi_{j,m}(y)\right)=\frac{1}{2\pi}\int_{0}^{2\pi}\prod_{k=1}^{2n}\sin\left(\dfrac{y-x_{j+k}}{2}\right)\textrm{d}y.
$$
Combining all these identities gives
$$
1=\sum_{j=l+1-m}^{r}\trig{w}_{j,m}\sin\left(\dfrac{x_{j+m}-x_j}{2}\right) T_{j,m}(x),
$$
where
$\trig{w}_{j,m}$ is given by \eqref{eq:B-spline-trig-weights-int}.

\bibliographystyle{plain}

\begin{thebibliography}{99}

\bibitem{Aigner:2007}
M.~Aigner.
\textit{A Course in Enumeration},
Springer-Verlag (2007).

\bibitem{AlbrechtMPR:2023a}
G.~Albrecht, E.~Mainar, J.~M.~Pe\~na, B.~Rubio.
A new class of trigonometric B-spline curves.
\textit{Symmetry}
\textbf{15}, 1551 (2023).

\bibitem{AlbrechtMPR:2023b}
G.~Albrecht, E.~Mainar, J.~M.~Pe\~na, B.~Rubio.
A shape preserving class of two-frequency trigonometric B-spline curves.
\textit{Symmetry}
\textbf{15}, 2041 (2023).

\bibitem{BangertP:1997}
C.~Bangert and H.~Prautzsch.
Circle and sphere as rational splines.
\textit{Neural, Parallel \& Scientific Computations}
\textbf{5}, 153--162 (1997).

\bibitem{BeccariGM:2020}
C.~V.~Beccari, G.~Casciola, and M.-L.~Mazure.
Critical length: An alternative approach.
\textit{Journal of Computational and Applied Mathematics}
\textbf{370}, 112603 (2020).

\bibitem{Boor:2001}
C.~de~Boor.
\textit{A Practical Guide to Splines}, Revised Edition,
Springer-Verlag (2001).

\bibitem{CarnicerMP:2003}
J.~M.~Carnicer, E.~Mainar, and J.~M.~Pe\~na.
Critical length for design purposes and extended Chebyshev spaces.
\textit{Constructive Approximation}
\textbf{20}, 55--71 (2003).

\bibitem{CohenRE:2001}
E.~Cohen, R.~F.~Riesenfeld, and G.~Elber.
\textit{Geometric Modeling with Splines: An Introduction},
CRC Press (2001).

\bibitem{CottrellHB:2009}
J.~A.~Cottrell, T.~J.~R.~Hughes, and Y.~Bazilevs.
\textit{Isogeometric Analysis: Toward Integration of CAD and FEA},
John Wiley \& Sons (2009).

\bibitem{Han:2015}
X.~Han.
Normalized {B}-basis of the space of trigonometric polynomials and curve design.
\textit{Applied Mathematics and Computation}
\textbf{251}, 336--348 (2015).

\bibitem{HiemstraHMST:2020}
R.~R.~Hiemstra, T.~J.~R.~Hughes, C.~Manni, H.~Speleers, and D.~Toshniwal.
A {T}chebycheffian extension of multidegree {B}-splines: Algorithmic computation and properties.
\textit{SIAM Journal on Numerical Analysis}
\textbf{58}, 1138--1163 (2020).

\bibitem{KochLNS:1995}
P.~E.~Koch, T.~Lyche, M.~Neamtu, and L.~L.~Schumaker.
Control curves and knot insertion for trigonometric splines.
\textit{Advances in Computational Mathematics}
\textbf{3}, 405--424 (1995).

\bibitem{LiS:2024}
M.~Li and W.-Q.~Shen.
Normalization for non-uniform trigonometric B-spline basis.
Manuscript (2024).

\bibitem{LycheSS:1998}
T.~Lyche, L.~L.~Schumaker, and S.~Stanley.
Quasi-interpolants based on trigonometric splines.
\textit{Journal of Approximation Theory}
\textbf{95}, 280--309 (1998).

\bibitem{LycheW:1979}
T.~Lyche and R.~Winther.
A stable recurrence relation for trigonometric {B}-splines.
\textit{Journal of Approximation Theory}
\textbf{25}, 266--279 (1979).

\bibitem{MainarPS:2001}
E.~Mainar, J.~M.~Pe\~{n}a, J. S\'anchez-Reyes.
Shape preserving alternatives to the rational {B}\'ezier model.
\textit{Computer Aided Geometric Design}
\textbf{18}, 37--60 (2001).

\bibitem{Mazure:2005}
M.~L.~Mazure.
Chebyshev spaces and Bernstein bases.
\textit{Constructive Approximation}
\textbf{22}, 347--363 (2005).

\bibitem{NeamtuPS:1998}
M.~Neamtu, H.~Pottmann, and L.~L.~Schumaker.
Designing {NURBS} cam profiles using trigonometric splines.
\textit{Journal of Mechanical Design}
\textbf{120}, 175--180 (1998).

\bibitem{PieglT:2012}
L.~Piegl and W.~Tiller.
\textit{The {NURBS} Book},
Springer-Verlag (2012).

\bibitem{RavalMS:2023}
K.~Raval, C.~Manni, and H.~Speleers.
Tchebycheffian B-splines in isogeometric Galerkin methods.
\textit{Computer Methods in Applied Mechanics and Engineering}
\textbf{403}, 115648 (2023).

\bibitem{SanchezReyes:1998}
J.~S\'anchez-Reyes.
Harmonic rational {B}\'ezier curves, $p$-{B}\'ezier curves and trigonometric polynomials.
\textit{Computer Aided Geometric Design}
\textbf{15}, 909--923 (1998).

\bibitem{SandeMS:2019}
E.~Sande, C.~Manni, and H.~Speleers.
Sharp error estimates for spline approximation: Explicit constants, $n$-widths, and eigenfunction convergence.
\textit{Mathematical Models and Methods in Applied Sciences}
\textit{29}, 1175--1205 (2019).

\bibitem{Schoenberg:1964}
I.~J.~Schoenberg.
On trigonometric spline interpolation.
\textit{Journal of Mathematics and Mechanics}
\textbf{13}, 795--825 (1964).

\bibitem{Schumaker:1982}
L.~L.~Schumaker.
On recursions for generalized splines.
\textit{Journal of Approximation Theory}
\textbf{36}, 16--31 (1982).

\bibitem{Schumaker:1983}
L.~L.~Schumaker.
On hyperbolic splines.
\textit{Journal of Approximation Theory}
\textbf{38}, 144--166 (1983).

\bibitem{Schumaker:2007}
L.~L.~Schumaker.
\textit{Spline Functions: Basic Theory}, Third Edition,
Cambridge University Press (2007).

\bibitem{ShenW:2005}
W.-Q.~Shen and G.-Z.~Wang.
A class of quasi {B}\'ezier curves based on hyperbolic polynomials.
\textit{Journal of Zhejiang University Science A}
\textbf{6}, 116--123 (2005).

\bibitem{Speleers:2022}
H.~Speleers.
Algorithm 1020: Computation of multi-degree Tchebycheffian B-splines.
\textit{ACM Transactions on Mathematical Software}
\textbf{48}, 12 (2022).

\bibitem{SpeleersT:2021}
H.~Speleers and D.~Toshniwal.
A general class of $C^1$ smooth rational splines: Application to construction of exact ellipses and ellipsoids.
\textit{Computer-Aided Design}
\textbf{132}, 102982 (2021).

\bibitem{Walz:1997}
G.~Walz.
Identities for trigonometric {B}-splines with an application to curve design.
\textit{BIT}
\textbf{37}, 189--201 (1997).

\end{thebibliography}

\end{document}